\documentclass[12pt, notitlepage]{article}   	
\usepackage[margin=1in]{geometry}                		
\usepackage{graphicx}				
\usepackage{amssymb}
\usepackage{amsmath}
\usepackage{amsthm}
\usepackage[T1]{fontenc}
\usepackage{mathptmx}
\usepackage{enumitem}

\newtheorem{theorem}{Theorem}
\newtheorem{lemma}{Lemma}[subsection]

\newtheorem{prop}{Proposition}[subsection]

\newtheorem{case}{Case}
\makeatletter
\@addtoreset{case}{lemma}
\makeatother

\DeclareMathOperator{\Sylv}{Sylv}

\usepackage{setspace}

\title{The Lower Central Series of a Quotient of a Free Algebra}
\author{\small Lev Kendrick, Gus Lonergan}
\date{}

\begin{document}

\maketitle
\thispagestyle{empty}
\begin{abstract}
Let $L_i(R)$ denote the $i^{\text{th}}$ term of the lower central series of an associative algebra $R$, and let $B_i(R)=L_i(R)/L_{i+1}(R)$.  We show that $B_2(\mathbb{C}\langle x, y\rangle/ P)\cong \Omega^2((\mathbb{C}\langle x, y\rangle/ P)_{ab})$, for all homogeneous or quasihomogeneous $P$ with square-free abelianization. Our approach generalizes that of Balagovic and Balasubramanian in (M. Balagovic, A. Balasubramanian, 2011), which in turn developed from that of Dobrovolska, Kim, and Ma (G. Dobrowolska, J. Kim, X. Ma, 2008). We also use ideas of Feigin and Shoikhet (B. Feign, B. Shoikhet, 2006), who initiated the study of the groups $B_i(R)$.
\end{abstract}

\setcounter{page}{1}
\section{Introduction}\label{sec:intro}

For an associative algebra $A$, we may consider the lower central series $L_i(A)$ given by $L_1(A)=A$ and $L_{i+1}(A)=[A,L_i(A)]$, and its subquotients $B_i(A)=L_i(A)/L_{i+1}(A)$. We would like to understand the structure of the spaces $B_i(A)$.

The spaces $B_i(A)$ were first studied by Feigin and Shoikhet (B. Feigin, B. Shoikhet, 2006), where they consider the case $A=A_n:=\mathbb{C}\langle x_1,\ldots,x_n\rangle$, the free algebra over $\mathbb{C}$. In particular they construct an isomorphism from $B_2(A_n)$ to the space of closed positive even differential forms on $\mathbb{A}^n_\mathbb{C}=\operatorname{Spec}((A^n_\mathbb{C})_{ab})$. Dobrovolska, Kim, and Ma (G. Dobrowolska, J. Kim, X. Ma, 2008) gave a new calculation of the Hilbert-Poincar\'e series of $B_2(A_n)$, establishing several general lemmas on which we rely heavily. In (M. Balagovic, A. Balasubramanian, 2011) Balagovic and Balasubramanian studied $B_2(A_n/P)$ for generic homogeneous $P$ and $n=2,3$, showing that it is isomorphic to the space of closed positive even differential forms on $\operatorname{Spec}((A_n/P)_{ab})$.

In this paper, we extend the result of Balagovic and Balasubramanian (for $n=2$) to any quasihomogeneous $P$ with square-free abelianization. Indeed recall that for any associative algebra $A$ we may define its abelianization as the maximal abelian quotient $A_{ab}$ of $A$, and for any $P\in A$ define its abelianization as its image $P_{ab}$ in $A_{ab}$. Recall also that $P\in \mathbb{C}\langle x,y\rangle\cong A_2$ is said to be quasihomogeneous if we can assign (positive integer) degrees $s,r$ to $x,y$ respectively such that $P$ is homogeneous, say of degree $d$, for the corresponding $\mathbb{Z}^+$-grading on $\mathbb{C}\langle x,y\rangle$. Thus we fix such $r,s$, without loss of generality coprime, and consider the corresponding grading on $\mathbb{C}\langle x,y\rangle$; unless otherwise stated, degrees (and related notions such as homogeneity, Hilbert-Poincar\'e series) are with respect to this grading.

Then we prove the following:

\begin{theorem}\label{thm:main}
Let $P\in\mathbb{C}\langle x, y\rangle$ be homogeneous, and such that the abelianization of $P$ is square-free. Then $B_2(\mathbb{C}\langle x, y\rangle/ P)$ is isomorphic to $\Omega^2((A_2/P)_{ab})$. In particular, its Hilbert-Poincar\'e series $HP(B_2(\mathbb{C}\langle x, y\rangle/ P))(t)$ is equal to $\frac{(t^{s}-t^{d})(t^{r}-t^{d})}{(1-t^{s})(1-t^{r})}$.\end{theorem}

This result is proved in the case of generic homogeneous $P$ in (M. Balagovic, A. Balasubramanian, 2011); thus even in the case of homogeneous $P$, this is a new result. Note that $\Omega^i((A_2/P)_{ab})=0$ for $i\geq 3$, so in particular $\Omega^2((A_2/P)_{ab})$ is precisely the space of closed positive even differential forms on $\operatorname{Spec}((A_2/P)_{ab})$. One may easily check that \begin{align*} \Omega^2((A_2/P)_{ab})\cong \mathbb{C}[x,y]\Bigg/\left(P_{ab},\frac{\partial P_{ab}}{\partial x},\frac{\partial P_{ab}}{\partial y}\right)dx\wedge dy.\end{align*} We have \begin{align}\label{eqn1} P_{ab}=d^{-1}\left(sx\frac{\partial P_{ab}}{\partial x}+ry\frac{\partial P_{ab}}{\partial y}\right),\end{align}
so that \begin{align*}\Omega^2((A_2/P)_{ab})\cong \mathbb{C}[x,y]\Bigg/\left(\frac{\partial P_{ab}}{\partial x},\frac{\partial P_{ab}}{\partial y}\right)dx\wedge dy.\end{align*} Since $\mathbb{C}[x,y]$ is a UFD, the sequence $(\partial P_{ab}/\partial x,\partial P_{ab}/\partial y)$ is regular if and only $\partial P_{ab}/\partial x,\partial P_{ab}/\partial y$ have no common prime factor. By equation (\ref{eqn1}), any such common factor $f$ would also divide $P_{ab}$, so that $P_{ab}=fg$ for some $g$ not divisible by $f$ (since $P_{ab}$ is assumed to be square-free). But then $f$ must divide $\partial P_{ab}/\partial x=f\partial g/\partial x+g\partial f/\partial x$ and therefore must also divide $\partial f/\partial x$, which is impossible for reasons of degree. So the sequence $(\partial P_{ab}/\partial x,\partial P_{ab}/\partial y)$ is regular, and \begin{align}\label{eqn2} HP(\mathbb{C}[x,y]/(\partial P_{ab}/\partial x,\partial P_{ab}/\partial y)dx\wedge dy) = \frac{(1-t^{d-s})(1-t^{d-r})}{(1-t^{s})(1-t^{r})}t^st^r = \frac{(t^{s}-t^{d})(t^{r}-t^{d})}{(1-t^{s})(1-t^{r})}.\end{align}
which explains the second sentence in Theorem \ref{thm:main} above. 

In Section \ref{sec:low} we adapt the map $B_2(A_n)\to \Omega^{\operatorname{even}}((A_n)_{ab})$ given (B. Feigin, B. Shoikhet, 2006) to construct a graded homomorphism $\phi:B_2(\mathbb{C}\langle x, y\rangle/ P)\to \Omega^2((A_2/P)_{ab})\cong \mathbb{C}[x,y]/(\partial P_{ab}/\partial x,\partial P_{ab}/\partial y)dx\wedge dy$, and show that it is surjective. This shows that \begin{align}\label{eqn3}HP(\mathbb{C}[x,y]/(\partial P_{ab}/\partial x,\partial P_{ab}/\partial y)dx\wedge dy)\leq HP(B_2(\mathbb{C}\langle x, y\rangle/ P)).\end{align}
with equality holding if and only if $\phi$ is an isomorphism. (Here given two power series $f(t),g(t)$ with integer coefficients we write $f\leq g$ to mean that for all $n$, the coefficient of $t^n$ in $f$ is less than or equal to the coefficient of $t^n$ in $g$.)

In Section \ref{sec:prelims} we adapt some results from (G. Dobrowolska, J. Kim, X. Ma, 2008) and (M. Balagovic, A. Balasubramanian, 2011) to the present situation, and use them in Section \ref{sec:series} to prove that \begin{align}\label{eqn4}HP(B_2(\mathbb{C}\langle x, y\rangle/ P))\leq \frac{(t^{s}-t^{d})(t^{r}-t^{d})}{(1-t^{s})(1-t^{r})}.\end{align}
Combining \ref{eqn2}, \ref{eqn3}, \ref{eqn4} shows that equality holds in \ref{eqn3}, \ref{eqn4}, and so the map $\phi$ must be an isomorphism; this is Theorem \ref{thm:main}.


\section{The Map $\phi$}\label{sec:low}

We follow (B. Feigin, B. Shoikhet, 2006). Let $(R,.,d)$ be a DG-algebra and let $R^{ev}$ be its even part; consider the bilinear operation $\star$ on $R^{ev}$ given by $a\star b=a.b+\frac{1}{2}da.db$. Simple calculations give: 

\begin{itemize}\item $(R^{ev},\star)$ is an associative algebra, with a decreasing filtration given by $F_i(R^{ev})=\bigoplus_{j\geq i}R^{2j}$.\
\item If $R^{ev}$ is central inside $(R,.)$ then $L_3((R^{ev},\star))=0$.\
\item If more specifically $(R,.)$ is a superalgebra (compatible with $d$) then we have $a\star b-b\star a =da.db$.\end{itemize}
It follows that the algebra homomorphism $A_n=\mathbb{C}\langle x_1,\dots,x_n\rangle\to (\Omega^{ev}(\mathbb{C}[x_1,\ldots,x_n]),\star)$ sending $x_i\mapsto x_i$ for each $i$ factors through $A_n/L_3(A_n)$, and we get by restriction a map $\phi_n: B_2(A_n)\to \Omega^{ev,+}(\mathbb{C}[x_1,\ldots,x_n])$. In (B. Feigin, B. Shoikhet, 2006), the following is shown:

\begin{lemma}\label{FSlem}$\phi_n: B_2(A_n)\to \Omega^{ev,+}(\mathbb{C}[x_1,\ldots,x_n])$ is injective, with image precisely those forms in $\Omega^{ev,+}(\mathbb{C}[x_1,\ldots,x_n])$ which are closed.\end{lemma}

We would like to produce analogous maps $\phi: B_2(A)\to \Omega^{ev,+}(A_{ab})$ for more general associative algebras $A$. In the case $A=\mathbb{C}\langle x,y\rangle/P$ we have the following construction.

Consider the algebra homomorphism $\psi:\mathbb{C}\langle x,y\rangle\to (\Omega^{ev}(\mathbb{C}[x,y]),\star)$ sending $x\mapsto x,y\mapsto y$ as above. Write $\Omega^{ev}(\mathbb{C}[x,y])=\mathbb{C}[x,y]\oplus \Omega^2(\mathbb{C}[x,y])$ and consider the corresponding decomposition $\psi=\psi_0+\psi_2$. One may easily check (for instance by induction on the length of a monomial in $\mathbb{C}\langle x,y\rangle$) that $\psi_0$ is the abelianization map.

Let's compute the image under $\psi$ of $\langle P\rangle\cap L_2(\mathbb{C}\langle x,y\rangle)$. Take an element of $\langle P\rangle\cap L_2(\mathbb{C}\langle x,y\rangle)$, and write it as $\sum_ia^iPb^i\in L_2(\mathbb{C}\langle x,y\rangle)$ for some $a^i,b^i\in\mathbb{C}\langle x,y\rangle$. Then its abelianization $\sum_ia^i_{ab}b^i_{ab}P_{ab}$ is equal to zero, and so $\sum_ia^i_{ab}b^i_{ab}=0$, since $P_{ab}$ is square-free and in particular nonzero. Dropping the subscript $_{ab}$, we compute: \begin{align*}\psi(\sum_ia^iPb^i) & =\sum_i(a^i+\psi_2(a^i))\star(P+\psi_2(P))\star(b^i+\psi_2(b^i))\\
& = \sum_i(a^i\star P\star b^i+\psi_2(a^i)\star P\star b^i+a^i\star \psi_2(P)\star b^i+a^i\star P\star \psi_2(b^i))\\
& = \sum_i(a^i\star P\star b^i+P.(a^i.\psi_2(b^i)+b^i.\psi_2(a^i))+a^i.b^i.\psi_2(P))\\
& = \sum_i(a^i\star P\star b^i+P.(a^i.\psi_2(b^i)+b^i.\psi_2(a^i)))\\
& = \sum_i(a^i.b^i.P +\frac{1}{2}(P.da^i.db^i+a^i.dP.db^i+b^i.da^i.dP)\\
& ~~~~~~~~~~~~~~~~~~~~~~~~~~~~~~~~~~~~~~~  +P.(a^i.\psi_2(b^i)+b^i.\psi_2(a^i)))\\
& = \sum_i(P.(a^i.\psi_2(b^i)+b^i.\psi_2(a^i)+\frac{1}{2}da^i.db^i) + \frac{1}{2}dP.(a^i.db^i-b^i.da^i))\end{align*}
We have used the symbol $.$ in place of $\wedge$ for consistency, but it should not be confusing. The key point is that all terms involving $\psi_2(P)$ vanish, and so the result lies in the (two-sided) $.$-ideal generated by $P_{ab}$ and $dP_{ab}$. It follows that the restriction $\psi:L_2(\mathbb{C}\langle x,y\rangle)\to \Omega^2(\mathbb{C}[x,y])\to \Omega^2((\mathbb{C}\langle x,y\rangle/P)_{ab})$ factors through $L_2(\mathbb{C}\langle x,y\rangle)/\langle P\rangle\cap L_2(\mathbb{C}\langle x,y\rangle)=L_2(\mathbb{C}\langle x,y\rangle/P)$, and it is automatic that it factors further through $B_2(\mathbb{C}\langle x,y\rangle/P)$. Let $\phi$ denote the resulting map $B_2(\mathbb{C}\langle x,y\rangle/P)\to   \Omega^2((\mathbb{C}\langle x,y\rangle/P)_{ab})$.

\begin{lemma}
$\phi:B_2(\mathbb{C}\langle x,y\rangle/P)\to \Omega^2((\mathbb{C}\langle x,y\rangle/P)_{ab})$ is surjective.\end{lemma}
\begin{proof}Since all forms in $\Omega^{ev,+}(\mathbb{C}[x,y])=\Omega^2(\mathbb{C}\langle x,y\rangle_{ab})$ are closed, Lemma \ref{FSlem} shows that $\phi_2:B_2(\mathbb{C}\langle x,y\rangle)\to \Omega^2(\mathbb{C}\langle x,y\rangle_{ab})$ is an isomorphism, and in particular surjective. $\phi$ is a certain quotient of $\phi$ by construction, and hence also surjective.
\end{proof}


\section{Auxiliary Results}\label{sec:prelims}


We begin with a lemma on $B_2(A)$ for an associative algebra $A$; it essentially exists in (M. Balagovic, A. Balasubramanian, 2011) already.

\begin{lemma}\label{lem:b2rels}
For any $q_1$, \ldots, $q_n$, $Q$, $a$, $b$, $c\in A$, the following hold in $B_2(A)$:

\begin{enumerate}[label=(\roman*)]
\item\label{eq:b2rels1} $[Q, q_1q_2\ldots q_n]=\sum_{i=1}^{n}[q_{i+1}q_{i+2}\ldots q_nQq_1q_2\ldots q_{i-1}, q_i]$.
\item\label{eq:b2rels2} $[ab, c]=[ba, c]$.
\item\label{eq:b2rels3} $(l+k)[a^lb,a^k]=k[b, a^{l+k}]$.
\item\label{eq:b2rels4} $[a^{i_1}q_1a^{i_2}q_2\ldots a^{i_n}q_n, a^i]=[a^{\sum_ki_k}q_1q_2\ldots q_n, a^i]$.
\end{enumerate}

\end{lemma}

\begin{proof}
c.f. Lemma 2.5 of (M. Balagovic, A. Balasubramanian, 2011). The lemma is for any associative ring $A$ (no assumption on commutative base ring). Since each identity in the lemma is inherited by quotients, it is enough to treat the case of a free associative algebra over $\mathbb{Z}$. \ref{lem:b2rels}.\ref{eq:b2rels1} holds in $L_2(A)$, by direct calculation. \ref{lem:b2rels}.\ref{eq:b2rels2} is clear. By \ref{lem:b2rels}.\ref{eq:b2rels1}, we have that $[a^lb,a^k]=\sum_{i=1}^{k}[a^{k-i}a^lba^{i-1}, a]$ which by \ref{lem:b2rels}.\ref{eq:b2rels2} is equal to $k[a^{l+k-1}b, a]$. So we have \begin{align*}[a^lb,a^k]=k[a^{l+k-1}b, a].\end{align*} It follows that $[b,a^{l+k}]=(l+k)[a^{l+k-1}b, a]$, whence $(l+k)[a^lb,a^k]=k(l+k)[a^{l+k-1}b, a]=k[b,a^{l+k}]$, which is \ref{lem:b2rels}.\ref{eq:b2rels3}.\\

\noindent It remains to show \ref{lem:b2rels}.\ref{eq:b2rels4}. By \ref{lem:b2rels}.\ref{eq:b2rels3}, we have $(i+i_1)[a^{i_1}q_1a^{i_2}q_2\ldots a^{i_n}q_n, a^i]=i[q_1a^{i_2}q_2\ldots a^{i_n}q_n, a^{i+i_1}]$, which by \ref{lem:b2rels}.\ref{eq:b2rels3} is equal to $i[a^{i_2}q_2\ldots a^{i_n}q_nq_1, a^{i+i_1}]$. Thus we have \begin{align*} (i+i_1)[a^{i_1}q_1a^{i_2}q_2\ldots a^{i_n}q_n, a^i]=i[a^{i_2}q_2\ldots a^{i_n}q_nq_1, a^{i+i_1}].\end{align*} Continuing, we have \begin{align*} (i+i_1+i_2)(i+i_1)[a^{i_1}q_1a^{i_2}q_2\ldots a^{i_n}q_n, a^i]=(i+i_1)i[a^{i_3}q_3\ldots a^{i_n}q_nq_1q_2, a^{i+i_1+i_2}]\end{align*} and eventually \begin{align*} (i+i_1+\ldots i_n)\ldots(i+i_1+i_2)(i+i_1)[a^{i_1}q_1a^{i_2}q_2\ldots a^{i_n}q_n, a^i]=(i+i_1+\ldots i_{n-1})\ldots(i+i_1)i[q_1q_2\ldots q_n, a^{i+\sum_k i_k}].\end{align*} One final application of \ref{lem:b2rels}.\ref{eq:b2rels3} to the RHS gives \begin{align*} (i+i_1+\ldots i_n)\ldots(i+i_1+i_2)(i+i_1)\left([a^{i_1}q_1a^{i_2}q_2\ldots a^{i_n}q_n, a^i]-[a^{\sum_ki_k}q_1q_2\ldots q_n, a^i]\right)=0.\end{align*} This completes the proof when $A$ is free over $\mathbb{Z}$, and so in every case.
\end{proof}

It follows (as in (G. Dobrowolska, J. Kim, X. Ma, 2008)) that $B_2(\mathbb{C}\langle x, y\rangle)$ is spanned by $[x^i,y^j]$ for $i,j\geq 1$. This also holds for its quotient $B_2(\mathbb{C}\langle x, y\rangle/P)$. This gives $\frac{t^{s+r}}{(1-t^s)(1-t^r)}$ a (bad) upper bound for the Hilbert-Poincar\'e series of $B_2(\mathbb{C}\langle x, y\rangle/P)$; to improve this upper bound we wish to find some (linear) relations between the elements $[x^i,y^j]$.



Recall that $P$ is (quasi-)homogeneous of degree $d$, and $P_{ab}$ is square-free. Since $r,s$ are coprime, we may write $P_{ab}=x^uy^v\sum_{k=0}^{n}a_kx^{(n-k)r}y^{ks}$ for some $n$, and some $u,v\in\{0,1\}$ with $d=us+vr+nrs$, and some $a_k\in\mathbb{C}$, with $a_0,a_n\neq 0$.




Let us write $u_k=u+(n-k)r$, $v_k=v+ks$, so that $P_{ab}=\sum_{k=0}^{n}a_kx^{u_k}y^{v_k}$. The following lemma gives relations among the elements $[x^i, y^j]$ spanning $B_2(\mathbb{C}\langle x, y\rangle/P)$ that allow us to improve the bound on its Hilbert-Poincar\'e series in Section \ref{sec:series}.

\begin{prop}\label{prop:bracketdep}
If $i\ge 0$ and $j\ge 1$, then, in $B_2(\mathbb{C}\langle x, y\rangle/\langle P\rangle)$,
$$\sum_{k=0}^{n}\frac{a_k}{j+v_k}[x^{i+u_k}, y^{j+v_k}]=0.$$
Similarly, if $i\ge 1$ and $j\ge 0$, then $\sum_{k=0}^{n}\frac{a_k}{i+u_k}[x^{i+u_k}, y^{j+v_k}]=0$.
\end{prop}
\begin{proof}
This is immediate from \ref{lem:b2rels}. Note in particular that \ref{lem:b2rels}.\ref{eq:b2rels4} implies that $[Q,y^j]$ depends only on $Q^{ab}$.
\end{proof}


\noindent Let us write $P_{ab}=x^uy^vQ$ where $Q=\sum_{k=0}^{n}a_kx^{(n-k)r}y^{ks}$. Then $Q$ is also square-free.

\begin{lemma}\label{lem:sqfree}
$\sum_{k=0}^{n}a_kx^{n-k}\in\mathbb{C}[x]$ is square-free. Symmetrically, $\sum_{k=0}^{n}a_kx^k\in\mathbb{C}[x]$ is also square-free.
\end{lemma}
\begin{proof}
Write $\sum_{k=0}^{n}a_kx^{n-k}\in\mathbb{C}[x]=a_0(x-\alpha_1)(x-\alpha_2)\ldots(x-\alpha_n)$, for $\alpha_i\in\mathbb{C}$.  Then $Q=a_0(x^r-\alpha_1y^s)(x^r-\alpha_2y^s)\ldots(x-\alpha_my^s)$, so the $\alpha_i$ are all distinct.
\end{proof}

\section{Bounding the Hilbert-Poincar\'e Series}\label{sec:series}

Recall we are trying to show that \begin{align}\label{eqn4}HP(B_2)\leq \frac{(t^{s}-t^{d})(t^{r}-t^{d})}{(1-t^{s})(1-t^{r})},\end{align} where $B_2$ denotes $B_2(\mathbb{C}\langle x, y \rangle/ P)$.  We only consider cases where $n\ge 1$: otherwise $\mathbb{C}\langle x, y\rangle/ P$ is either $\mathbb{C}[x]$ or $\mathbb{C}[y]$, and Theorem \ref{thm:main} clearly holds.  All equations involving the indeterminates $x$ and $y$ are either in $B_2$ or $\mathbb{C}[x, y]$; it will always be clear from context which of the two spaces is being used.

For $m\geq 0$, let $S_m=\{(i,j)\in\mathbb{Z}_{\geq 0}^{2}|si+rj=m\}$ and $S_m^+=\{(i,j)\in\mathbb{Z}_{> 0}^{2}|si+rj=m\}$, and denote by $B_2[m]$ the degree-$m$ part of $B_2$. We re-write an earlier observation in the following proposition.
\begin{prop}\label{prop:nullbasis}
$\dim B_2[m]\le |S^{+}_m|$.
\end{prop}

The next subsections mostly focus on cases where $S_{m-d}\neq\emptyset$.  In such cases, let $S_{m-d}=\{(p+(l-t)r, q+ts)\mid 0\le t\le l\}$, for some $l$, where $0\le p\le r-1$ and $0\le q\le s-1$.  Let $(p_t, q_t)$ denote $(p+(l-t)r, q+ts)$. Note that $S_m^+\subset \{(u+p+(n+l-i)r,v+q+is)| 0\leq i\leq n+l\}$, with equality holding unless $u+p=0$ or $v+q=0$. (This is because $u+p\leq r$, $v+q\leq s$.)

For convenience, we restate Proposition \ref{prop:bracketdep} in a slightly different form.

\begin{prop}\label{prop:bracketdep2}
For $t\in\{0,\ldots, l\}$, $p_t>0$ implies $\sum_{k=0}^{n}\frac{a_k}{u_k+p_t}[x^{u_k+p_t}, y^{v_k+q_t}]=0$ and $q_t>0$ implies $\sum_{k=0}^{n}\frac{a_k}{v_k+q_t}[x^{u_k+p_t}, y^{v_k+q_t}]=0$.
\end{prop}

We divide the remainder of the proof into three cases: Subsection \ref{subsec:nrs} deals with the case $(u,v)=(0,0)$, Subsection \ref{subsec:nrsrs} with $(u,v)=(1,1)$, and Subsection \ref{subsec:nrss} with $(u,v)=(0,1)$, which is equivalent to the case $(u,v)=(1,0)$. We will call the top row of a matrix the $0^{th}$, and so on.


\subsection{$P_{\text{ab}}=\sum_{k=0}^{n}a_kx^{(n-k)r}y^{ks}$}\label{subsec:nrs}


\begin{lemma}\label{lem:dimbd}
If $S_{m-d}\neq\emptyset$, then $\dim B_2[m]\le \max(n-l-1, 0)$.
\end{lemma}
\begin{proof}
Let $\Sylv(f, g)$ denote the Sylvester matrix of $f$ and $g$, for $f$, $g\in\mathbb{C}[x]$.  The relevant property of the Sylvester matrix is that $\det(\Sylv(f, g))=\operatorname{Res}(f,g)$ is the resultant of $f$ and $g$, which is equal to zero if and only if $f$ and $g$ share a root.  We divide this proof into three cases: $p\neq 0$ and $q\neq 0$, $p=0$ and $q\neq 0$, and $p=q=0$ (the case where $p\neq0$ and $q=0$ is covered by the second case by symmetry).

\begin{case}\label{case:nrs1}
$p\neq 0$ and $q\neq 0$:
\end{case}
In this case, by Proposition \ref{prop:bracketdep2}, for $t\in\{0,\ldots, l\}$ we have
\begin{equation}\label{eq:lindeps1}
\begin{array}{rcl}
\sum_{k=0}^{n}\frac{a_k}{p+(n+l-k-t)r}[x^{p+(n+l-k-t)r}, y^{q+(k+t)s}] & = & 0 \\
\sum_{k=0}^{n}\frac{a_k}{q+(k+t)s}[x^{p+(n+l-k-t)r}, y^{q+(k+t)s}] & = & 0.
\end{array}
\end{equation}
This is an equation $AX=0$, where $X$ is the column vector with entries $[x^{p+(n+l-i)r}, y^{q+i s}]$ for $0\leq i\leq n+l$ and where
\begin{equation*}
A=
\begin{pmatrix}
\frac{a_0}{p+(n+l)r} & \frac{a_1}{p+(n+l-1)r} & \ldots & \frac{a_n}{p+lr} & 0 & \ldots & 0 \\
0 & \frac{a_0}{p+(n+l-1)r} & \ldots & \frac{a_{n-1}}{p+lr} & \frac{a_n}{p+(l-1)r} & \ldots & 0 \\
\vdots & \vdots & \ddots & \vdots & \vdots & \ddots & \vdots \\
\ldots & \ldots & \frac{a_0}{p+nr} & \ldots & \ldots & \ldots & \frac{a_n}{p} \\
\frac{a_0}{q} & \frac{a_1}{q+s} & \ldots & \frac{a_n}{q+ns} & 0 & \ldots & 0 \\
0 & \frac{a_0}{q+s} & \ldots & \frac{a_{n-1}}{q+ns} & \frac{a_n}{q+(n+1)s} & \ldots & 0 \\
\vdots & \vdots & \ddots & \vdots & \vdots & \ddots & \vdots \\
\ldots & \ldots & \frac{a_0}{q+ls} & \ldots & \ldots & \ldots & \frac{a_n}{q+(n+l)s}
\end{pmatrix}.
\end{equation*}
Recall that $B_2[m]$ is spanned by the elements $[x^i, y^j]$ for $(i, j)\in S^{+}_m$. These are precisely the $n+l+1$ entries of $X$. $A$ is a $(2l+2)\times(l+n+1)$ matrix, so if it has maximal rank it will follow that $\dim B_2[m]\le\max(n-l-1, 0)$. We prove that $A$ has maximal rank.

We begin by performing a series of row- and column-operations on $A$. First scale the $i^{\text{th}}$ column of $A$ by $(p+(n+l-i)r)(q+is)$ for $0\le i\le n+l$. Next, for $0\le j\le l$, scale the $j^{\text{th}}$ row by $r$; then add $s$ times the $(j+l+1)^{\text{th}}$ row to the $j^{\text{th}}$; then scale the $j^{\text{th}}$ row by $m^{-1}$; then subtract $p+lr$ times the $j^{\text{th}}$ row from the $(j+l+1)^{\text{th}}$ row; then scale the $(j+l+1)^{\text{th}}$ row by $r^{-1}$.

This produces the equivalent matrix \begin{equation*}
B=
\begin{pmatrix}
a_0 & a_1 & \ldots & a_n & 0 & \ldots & 0 \\
0 & a_0 & \ldots & a_{n-1} & a_n & \ldots & 0 \\
\vdots  & \vdots & \ddots & \vdots & \vdots & \ddots & \vdots \\
0 & \ldots & a_0 & \ldots & \ldots & \ldots & a_n \\
na_0 & (n-1)a_1 & \ldots & 0a_n & 0 & \ldots & 0 \\
0 & na_0 & \ldots & a_{n-1} & 0a_n & \ldots & 0 \\
\vdots  & \vdots & \ddots & \vdots & \vdots & \ddots & \vdots \\
0 & \ldots & na_0 & \ldots & \ldots & \ldots & 0a_n
\end{pmatrix}
\end{equation*}where the first $l+1$ rows contain $(a_0,\ldots,a_n)$ and the second $l+1$ rows contain $(na_0,\ldots,0a_n)$.


Let $h=a_0x^n+\ldots+a_n$, which is square-free thanks to Lemma \ref{lem:sqfree}. Thus the resultant of $h$ with $h'$ is nonzero. Now notice that if $l+1\leq n-1$, then it is possible to obtain $\operatorname{Sylv}(h,h')$ from $B$ first appending some columns of zeros to the right, and then inserting some rows. So in that case the rows are linearly independent, and $B$ (whence also $A$) has maximal rank. So suppose that $l+1\geq n$. Consider the $(n+l+1)\times (n+l+1)$ submatrix of $B$ consisting of the first $n+l+1$ rows; one may check that its determinant is equal to $\pm a_n^{l-n+2}\operatorname{Res}(h,h')\neq 0$; thus in this case the columns are linearly independent, and $B$ (whence also $A$) has maximal rank.

\begin{case}\label{case:nrs2}
$p=0$ and $q\neq 0$:
\end{case}
Here Proposition \ref{prop:bracketdep2} gives us $2l+1$ linear relations: two for every $(p_t, q_t)$ with $p_t$, $q_t>0$, and one for $(0, q+ls)$.  This system of equations involves $l+n$ unknowns, for the $l+n$ elements in $S^{+}_m$, giving a $(2l+1)\times(l+n)$ matrix equation, with a matrix $A$.  As in Case \ref{case:nrs1}, the result is established by showing the independence of the rows of $A$ for $l\le n-1$ and showing the existence of a nonzero $(l+n)\times(l+n)$ minor of $A$ for $l\ge n$.  The proof is entirely analogous to the previous case; we only comment that although the bottom row of $A$ (corresponding to the one relation for $t=l$) has no counterpart row, it can nonetheless be brought to $(0,\ldots,na_0,\ldots,a_{n-1})$ purely by row and column scaling. Thus $A$ is equivalent to \begin{equation*}
B=
\begin{pmatrix}
a_0 & a_1 & \ldots & a_n & 0 & \ldots & 0 \\
0 & a_0 & \ldots & a_{n-1} & a_n & \ldots & 0 \\
\vdots  & \vdots & \ddots & \vdots & \vdots & \ddots & \vdots \\
0 & \ldots & a_0 & \ldots & \ldots & \ldots & a_n \\
na_0 & (n-1)a_1 & \ldots & a_{n-1} & 0 & \ldots & 0 \\
0 & na_0 & \ldots & 2a_{n-2} & a_{n-1} & \ldots & 0 \\
\vdots  & \vdots & \ddots & \vdots & \vdots & \ddots & \vdots \\
0 & \ldots & na_0 & \ldots & \ldots & \ldots & a_{n-1}
\end{pmatrix}
\end{equation*} where the first $l$ rows contain $(a_0,\ldots,a_n)$ and the second $l+1$ rows contain $(na_0,\ldots,a_{n-1})$. Similar arguments demonstrate that it has maximal rank.

\begin{case}\label{case:nrs3}
$p=q=0$:
\end{case}
We play the same game. This time, $A$ is of size $(2l)\times (n+l-1)$ and is similarly equivalent to
\begin{equation*}B=
\begin{pmatrix}
a_0 & a_1 & \ldots & a_{n-1} & a_n & 0 & \ldots & 0 \\
0 & a_0 & \ldots & a_{n-2} & a_{n-1} & a_n & \ldots & 0 \\
\vdots & \vdots & \ddots & \vdots & \vdots & \vdots & \ddots & \vdots \\
0 & 0 & \ldots & a_0 & \ldots & \ldots & a_{n-1} & a_n\\
a_1 & 2a_2 & \ldots & na_n & 0 & 0 & \ldots & 0 \\
na_0 & (n-1)a_1 & \ldots & 0 & 0 & \ldots & 0 \\
0 & na_0 & \ldots & 2a_{n-2} &  a_{n-1} & 0 & \ldots & 0 \\
\vdots & \vdots & \ddots & \vdots & \vdots & \vdots & \ddots & \vdots \\
0 & \ldots & na_0 & \ldots & \ldots & \ldots & \ldots & a_{n-1}
\end{pmatrix}.
\end{equation*} where the first $l-1$ rows contain $(a_0,\ldots,a_n)$, the next (i.e. the $(l-1)^{th}$) row contains $(a_1,2a_2,\ldots,na_n)$ and the last $l$ rows contain $(na_0,\ldots,a_{n-1})$.

First, let $M$ denote the special case of $B$ where $l=n-1$. We claim that $M$ is nonsingular.  To see this, consider $\Sylv(h, h')$.  Subtracting $n$ times the $0^{\text{th}}$ row from the $(n-1)^{\text{th}}$ row of $\Sylv(h, h')$ and scaling the $(n-1)^{th}$ row by $-1$ yields a matrix whose determinant is given by $a_0\det M\neq 0$.

For $l\leq n-1$, we can obtain $M$ from $B$ by first adding some columns of zeros to the right, and then inserting some rows; it follows that the rows of $B$ are linearly independent. So assume $l\geq n$. As in the previous cases, one may check that the determinant of the square matrix given by the first $n+l$ rows, excluding the $(l-1)^{th}$ (that is the row starting $(a_1,2a_2,\ldots)$) is equal to $\pm a_n^{l-n}\operatorname{Res}(h,h')\neq 0$; thus the columns of $B$ are linearly independent.\end{proof}

We are now in a position to complete the proof of Theorem \ref{thm:main} in this case. When $S_{m-d}\neq\emptyset$, there are $\max(n-l-1, 0)$ elements $(i, j)\in S_m$ satisfying $1\le i\le nr-1$ and $1\le j\le ns-1$, namely $\{(p+(n+l-i)r,q+is)| l+1\leq i\leq n-1\}$. We remark that the corresponding elements of $B_2[m]$ do not necessarily span $B_2[m]$, but it does not matter. We conclude that \begin{align*}\dim B_2[m]\leq |S_m\cap\{(i, j)\in\mathbb{Z}^2\mid 1\le i\le nr-1 \text{ and } 1\le j\le ns-1\}|\end{align*} for all $m$.  We now see that the coefficients of the Hilbert-Poincar\'e series of $B_2$ are bounded above by the coefficients of $(t^r+t^{2r}+\ldots+t^{(ns-1)r})(t^s+t^{2s}+\ldots+t^{(nr-1)s})$, which is equal to $\frac{(t^{s}-t^{d})(t^{r}-t^{d})}{(1-t^{s})(1-t^{r})}$, as required ($d=nrs$ in this case).

\subsection{$P_{\text{ab}}=xy(\sum_{k=0}^{n}a_kx^{(n-k)r}y^{ks})$}\label{subsec:nrsrs}

The proof of Theorem \ref{thm:main} here is similar to its proof in Subsection \ref{subsec:nrs}.  As in Subsection \ref{subsec:nrs}, we assume $S_{m-d}\neq\emptyset$ and consider three cases: $p\neq 0$ and $q\neq 0$, $p=0$ and $q\neq 0$, and $p=q=0$. Again in each case we write $AX=0$ where $X$ contains the elements spanning $B_2[m]$ labeled by elements of $S_m^+$, and $A$ encodes all relations given by Proposition \ref{prop:bracketdep2}. In each case we use variations of the arguments found in Subsection \ref{subsec:nrs} to show that $A$ has maximal rank. We record the results, together with comments on the proofs, below.


\begin{lemma}\label{lem:rsdim1}
If $S_{m-d}\neq\emptyset$, $p\neq 0$ and $q\neq 0$, then $\dim B_2[m]\le\max(n-l-1, 0)$.
\end{lemma}
\begin{proof}
This is essentially identical to Case \ref{case:nrs1} of Lemma \ref{lem:dimbd}; we still have $2l+2$ linear equations in $n+l+1$ spanning elements; $A$ appears slightly different (all the denominators are greater by $1$) but it is equivalent by the same operations to the same matrix $B$ as in that case.\end{proof}

\begin{lemma}\label{lem:rsdim2}
If $S_{m-d}\neq\emptyset$, $p=0$, and $q\neq 0$, then $\dim B_2[m]\le\max(n-l, 0)$.
\end{lemma}
\begin{proof}
This is similar to the proof of Case \ref{case:nrs2} of Lemma \ref{lem:dimbd}.  The main difference is that $S^{+}_m$ had $n+l$ elements in Subsection \ref{subsec:nrs}, whereas here it has $n+l+1$ elements. So $A$ is of size $(2l+1)\times(n+l+1)$. It is equivalent to the matrix obtained from the matrix $B$ of Case \ref{case:nrs2} of Lemma \ref{lem:dimbd} by adjoining an extra column whose bottom entry is $1$ and whose other entries are $0$. This is easily seen to have maximal rank.\end{proof}

\begin{lemma}\label{lem:rsdim3}
If $S_{m-d}\neq\emptyset$ and $p=q=0$, then $\dim B_2[m]\le\max(n-l+1, 0)$.
\end{lemma}
\begin{proof}
This is similar to the proof of Case \ref{case:nrs3} of Lemma \ref{lem:dimbd}.  The main difference is that $S^{+}_m$ had $n+l-1$ elements in Subsection \ref{subsec:nrs}, whereas here it has $n+l+1$ elements.  $A$ is then of size $(2l)\times(n+l+1)$. It is equivalent to the matrix obtained from the matrix $B$ of Case \ref{case:nrs3} of Lemma \ref{lem:dimbd} by adjoining an extra column whose bottom entry is $1$ and whose other entries are $0$, and an extra column whose $(l-1)^{th}$ entry is $1$ and whose other entries are $0$. This is easily seen to have maximal rank.
\end{proof}

We are now in a position to complete the proof of Theorem \ref{thm:main} in this case. When $S_{m-d}\neq\emptyset$, it is easy to check that the number of pairs $(i, j)\in S_m$ satisfying $1\le i\le nr+1$ and $1\le j\le ns+1$ coincides with the upper bounds on $\dim B_2[m]$ given by Lemmas \ref{lem:rsdim1}, \ref{lem:rsdim2}, and \ref{lem:rsdim3}. So (recalling Proposition \ref{prop:nullbasis}) we have \begin{align*}\dim B_2[m] \leq |S_m\cap\{(i, j)\in\mathbb{Z}^2\mid 1\le i\le nr+1 \text{ and } 1\le j\le ns+1\}|\end{align*} for all $m$.  Thus the coefficients of the Hilbert-Poincar\'e series are bounded above by the coefficients of $(t^r+t^{2r}+\ldots+t^{(ns+1)r})(t^s+t^{2s}+\ldots+t^{(nr+1)s})$, which is equal to $\frac{(t^{s}-t^{d})(t^{r}-t^{d})}{(1-t^{s})(1-t^{r})}$, as required ($d=nrs+r+s$ in this case).


\subsection{$P_{\text{ab}}=y(\sum_{k=0}^{n}a_kx^{(n-k)r}y^{ks})$}\label{subsec:nrss}

The proof of Theorem \ref{thm:main} here is similar to its proof in Subsection \ref{subsec:nrs}. However, in this case, because we have lost symmetry between $x$ and $y$, we must consider four cases: $p\neq 0$ and $q\neq 0$, $p=0$ and $q\neq 0$, $p\neq 0$ and $q=0$, and $p=q=0$. 


\begin{lemma}\label{lem:sdim1}
If $S_{m-d}\neq\emptyset$, $p\neq 0$, and $q\neq 0$, then $\dim B_2[m]\le\max(n-l-1, 0)$.
\end{lemma}
\begin{proof}
This is essentially identical to the proof of Case \ref{case:nrs1} of Lemma \ref{lem:dimbd}.
\end{proof}

\begin{lemma}\label{lem:sdim2}
If $S_{m-d}\neq\emptyset$, $p=0$, and $q\neq 0$, then $\dim B_2[m]\le\max(n-l-1, 0)$.
\end{lemma}
\begin{proof}
This is essentially identical to the proof of Case \ref{case:nrs2} of Lemma \ref{lem:dimbd}.
\end{proof}

\begin{lemma}\label{lem:sdim3}
If $S_{m-d}\neq\emptyset$, $p\neq 0$, and $q=0$, then $\dim B_2[m]\le\max(n-l, 0)$.
\end{lemma}
\begin{proof}
This is essentially identical to the proof of Lemma \ref{lem:rsdim2}, or more strictly, the case $(u,v)=(1,1)$, $p\neq 0,q=0$ which may be proved as Lemma \ref{lem:rsdim2} by switching the roles of $x$ and $y$.
\end{proof}

\begin{lemma}\label{lem:sdim4}
If $S_{m-d}\neq\emptyset$ and $p=q=0$, then $\dim B_2[m]\le\max(n-l, 0)$.
\end{lemma}
\begin{proof}
This is midway between Case \ref{case:nrs3} of Lemma \ref{lem:dimbd} and Lemma \ref{lem:rsdim3}. The relevant matrix $A$ is equivalent to the matrix which is obtained from the matrix $B$ of Case \ref{case:nrs3} of Lemma \ref{lem:dimbd} by adjoining an extra column whose $(l-1)^{th}$ entry is $1$ and whose other entries are $0$. This is easily seen to have maximal rank.\end{proof}

We are now in a position to complete the proof of Theorem \ref{thm:main} in this case. When $S_{m-d}\neq\emptyset$, it is easy to check that the number of pairs $(i, j)\in S_{m}$ satisfying $1\le i\le nr-1$ and $1\le j\le ns+1$ coincides with the upper bounds on $\dim B_2[m]$ given by Lemmas \ref{lem:sdim1}, \ref{lem:sdim2}, \ref{lem:sdim3}, and \ref{lem:sdim4} , in all cases except $m=d$.  When $m=d$, that number is one less than the bound for $\dim B_2[d]$. We conclude that \begin{align*}\dim B_2[m]\leq |S_m\cap(\{(i, j)\in\mathbb{Z}^2\mid 1\le i\le nr-1 \text{ and } 1\le j\le ns+1\}\cup\{(nr,1)\})|.\end{align*}  Therefore the coefficients of the Hilbert-Poincar\'e series are bounded above by the coefficients of $t^{d}+(t^s+t^{2s}+\ldots+t^{(nr-1)s})(t^r+t^{2r}+\ldots+t^{(ns+1)r})$, which is equal to $\frac{(t^{s}-t^{d})(t^{r}-t^{d})}{(1-t^{s})(1-t^{r})}$, as required ($d=nrs+r$ in this case).

\section{Directions for Future Work}\label{sec:conc}

One might start with free algebras on more variables, consider quotients by multiple relations, and examine of $B_i$ for $i\ge3$; references (M. Balagovic, A. Balasubramanian, 2011), (G. Dobrowolska, J. Kim, X. Ma, 2008), and (B. Feigin, B. Shoikhet, 2006) discuss these, and it may be possible to adapt the methods of the current paper to these more general settings. One might also work over rings other than $\mathbb{C}$, such as $\mathbb{Z}$; $\mathbb{C}$ is sufficient for computing ranks, but in general the $B_i$ will have torsion. This is discussed by Cordwell, Fei and Zhou (K. Cordwell, T. Fei, K. Zhou, 2015), and the references therein.

\section{Acknowledgments}
The authors thank Prof. Pavel Etingof of MIT for introducing them to this area of research, suggesting the problem, and for his subsequent time and advice. The authors thank RSI and the MIT mathematics department for supporting them in this research.

\end{document}